\documentclass[10pt,english]{amsart}
\usepackage{graphicx} 
\usepackage[a4paper, margin=3cm]{geometry}
\usepackage[english]{babel}
\usepackage{amsmath}
\usepackage{amssymb}
\usepackage{amsthm}
\usepackage{graphicx}
\usepackage{hyperref}
\usepackage{tikz-cd}
\usepackage{adjustbox}

\title{A $p$-adic cohomological approach to congruences of Meromorphic Modular Forms}
\author{Paolo Bordignon}
\date{January 2026}
\email{p.bordignon@math.leidenuniv.nl}
\address{Mathematical Institute, Gorlaeus Building, Einsteinweg 55, 2333 CC, Leiden (Netherlands)}

\newtheorem{theo}{Theorem}
\newtheorem*{theo*}{Theorem}
\newtheorem{lemm}{Lemma}

\newtheorem{rema}{Remark}
\newtheorem{coro}{Corollary}

\newtheorem{mainthm}{Theorem}

\newtheorem{theorem}{Theorem}

\begin{document}

\begin{abstract}
We study congruences relating Fourier coefficients of meromorphic modular forms and Frobenius eigenvalues of elliptic curves corresponding to their poles. We develop a $p$-adic cohomological framework that interprets these congruences via the interaction between the rigid cohomology of modular curves and the crystalline structure of the associated elliptic curves. Using comparison theorems and the Gysin sequence, we relate the Frobenius actions in cohomology to the $U_p$-operator acting on spaces of overconvergent modular forms.
Our approach applies uniformly to both modular curves and Shimura curves admitting smooth integral models over $\mathbb{Z}_p$. 
\end{abstract}
\maketitle

\tableofcontents
\section{Introduction}
In this work, we provide a $p$-adic cohomological interpretation of certain congruences between Fourier coefficients of modular forms and Frobenius eigenvalues arising from the geometry of modular (and Shimura) curves, as studied in a recent work by Zhang \cite{zhang2025ellipticcurvesfouriercoefficients}. 
Zhang numerically observed several congruences of the form
\begin{equation*}
    a_p\left(\frac{E_4}{j-j(C)}\right)\equiv a_p(C)^2 \mod p 
\end{equation*}
where $C/\mathbb{Q}$ is an elliptic curve with good reduction at $p$ and $a_p(C)=p+1-|C(\mathbb{F}_p)|$. 
To study these relations, we introduce a cohomological framework and use the Gysin sequence to relate classes attached to modular forms to differential forms attached to the elliptic curves corresponding to the poles.

More precisely, let $X$ be the elliptic or quaternionic modular curve of level $N$, defined over $\mathbb{Z}_p$ with $(p,N)=1$. Consider $\omega$ the line bundle obtained from the universal object whose global sections are modular forms. Let $\mathcal{H}$ be the vector bundle of rank 2 obtained from the relative de Rham cohomology. To relate meromorphic modular forms with prescribed structure at the poles, we use de Rham cohomology with coefficients and relate classes in $H^1_\mathrm{dR}(X\setminus \{\alpha\},\mathcal{H})$ with fibers $\alpha^*\mathcal{H}$ for $\alpha$ a $\mathbb{Z}_p$-point. Using comparison theorems between rigid and de Rham cohomology, we relate meromorphic modular forms to overconvergent $p$-adic modular forms. In this way, we can naturally study the Frobenius action on the various cohomology groups. Enlarging the set of poles considered will allow us to study an explicit Frobenius action on the meromorphic modular forms representing the cohomology classes and to relate it with the action of the $U_p$-operator.
Theorem \ref{Theorem characteristic polynomial of Up action} is the core result of this paper and shows that the Frobenius eigenvalues of the de Rham cohomology at the poles are related to the $U_p$ action on meromorphic modular forms. 
\begin{theorem}
    Let $X$ be a modular curve (elliptic or quaternionic) of level $N$ defined over $\mathbb{Z}_p$, with $N>3$, $p>3$ and $(N,p)=1$. Fix a $\mathbb{Z}_p$-point $\alpha$ on the modular curve. Let $k$ be a positive integer with $p>k+1$, let $P(X),Q(X)\in \mathbb{Q}_p[X]$ be respectively the characteristic polynomials of the Frobenius action on $\mathrm{Sym}^k(\alpha^*\mathcal{H})[1]$ and on $H^1_{\mathrm{dR}}(X^{\mathrm{rig}},\mathcal{H}_k)$
    \begin{align*}
        P(X)=\det(\mathrm{Frob}-X&|\mathrm{Sym}^k(\alpha^*\mathcal{H})[1])=\sum_{i=0}^{k+1}c_iX^i,\\ Q(X)=\det(\mathrm{Frob}-X&|H^1_{\mathrm{dR}}(X^{\mathrm{rig}},\mathcal{H}_k)),\\
        P(X)Q(X)&=\sum_{i=0}^{d+k+1}e_iX^i
    \end{align*}
    with $d=\dim_{\mathbb{Q}_p}H^1_{\mathrm{dR}}(X^{\mathrm{rig}},\mathcal{H}_k)=\dim_{\mathbb{Q}_p}M_{k+2}+\dim_{\mathbb{Q}_p}S_{k+2}$. Let $M=d+k+1$, then for every $f\in M_{k+2}^{\mathrm{mero},\{\alpha\}}$ we have
    \begin{equation*}
        \sum_{i=0}^{M}e_{M-i}p^{M-i}U_p^i(f)\in \theta^{k+1}\left(M_{-k,\lambda}^{\dagger,\{\alpha\}}\right)
    \end{equation*}
    where $\theta$ is the differential operator $q\frac{\mathrm{d}}{dq}$ in the elliptic case and $M_{-k,\lambda}^{\dagger,\{\alpha\}}$ is the space of $\lambda$-overconvergent modular forms defined away from a small disk around $\alpha$.
\end{theorem}

The integrality provided by crystalline cohomology then allows us to deduce congruence relations. In the elliptic setting, we can derive an explicit congruence at the level of $q$-expansions as shown in Corollary \ref{congruence q-coefficients}. 

Some technical constructions are only briefly outlined in this article, such as the construction of the $U_p$ operator on the open curve and the extension of Coleman's result \ref{V and U_p inverse} to Shimura curves. A more detailed treatment and description will appear in the author’s PhD thesis.

\subsection{An important note on a recent result}
A recent preprint \cite{allen2025atkinswinnertondyercongruencesmeromorphic} came out during the writing of this article. The authors analyze the same congruences between meromorphic modular forms and elliptic curves and give a proof of many of Zhang's conjectures. Their methods extend the classical work of Scholl \cite{scholl1985modular} and a more recent work by Kazalicki-Scholl \cite{kazalicki‑scholl2013} to the case of meromorphic modular forms. The works are similar in many aspects, in particular the initial cohomological setting and the Gysin sequence are used in fundamentally the same way. Our approach differs from theirs primarily in the use of $\log$-crystalline cohomology to identify a Frobenius stable lattice and the use of an explicit Frobenius lift enlarging the set of poles to the set of supersingular points. This allows us to avoid computing with local expansions at the cusp at infinity, enlarge our framework to the Shimura setting and derive relations between Fourier coefficients only at the very end. On the other hand, their precise work takes care of the order of the poles of the meromorphic modular forms allowing them to obtain very sharp congruences on the cohomological subgroup. In particular, they obtain a stronger result in the CM case.

After a very kind and fruitful exchange of emails with the authors, where they have looked at a first draft of this work and at the techniques involved, they allowed us to publish our version. 

\subsection{Acknowledgments}
The author would like to thank Jan Vonk for his invaluable supervision and support, and for drawing his attention to many of the numerical congruences considered in this paper. He also thanks Pengcheng Zhang and Tiago J. Fonseca for inspiring talks during their visits to Leiden University.

\section{Geometric setting}
Let $p>3$ be a prime, let $N>3$ be an integer coprime to $p$. Consider one of the following settings

\begin{itemize}
    \item $X=X_1(N)$, the modular curve of level $\Gamma_1(N)$ defined over $\mathbb{Z}_p$ classifying elliptic curves with level structure,
    \item $X=X_B$ a Shimura curve defined over $\mathbb{Z}_p$ for $B$ an indefinite non-split quaternion algebra over $\mathbb{Q}$ of discriminant $\delta$ coprime with $p$ and of level $V_1(N)$ classifying false elliptic curves with level structure.
\end{itemize}
We will denote $\overline{X}=X\times \mathbb{F}_p$ its special fiber, let
$C\subset X(\mathbb{Z}_p)$ be the finite collection (possibly empty) of cusps. Let $\omega$ be the ample line bundle over $X$ obtained from the pushforward of the differential bundle on the universal object. Let $(\mathcal{H},\nabla)$ be the rank 2 vector bundle arising from the relative de Rham cohomology over $X$ with logarithmic connection
    \begin{equation*}
        \nabla:\mathcal{H}\rightarrow\mathcal{H}\otimes_{\mathbb{Z}_p}\Omega^1_{X}(\log C)
    \end{equation*}
    coming from the Gauss-Manin connection. In the quaternionic case we are fixing a choice of a non-trivial idempotent element of $M_2(\mathbb{Z}_p)$ and projecting the four-dimensional vector bundle coming from the relative de Rham cohomology to obtain $\mathcal{H}_k$ and analogously with $\omega$. See \cite{kassaei2004padic}. 
    Denote by $\mathcal{H}_k$ its $k$-th symmetric tensor power $\mathcal{H}_k=\mathrm{Sym}^k\mathcal{H}$.

We recall the standard properties of $\mathcal{H}_k$.
The vector bundles $\omega$ and $\mathcal{H}$ fit in a short exact sequence
    \begin{equation}\label{hodge filtration 1}
    0\rightarrow \omega\rightarrow \mathcal{H}\rightarrow \omega^{-1}\rightarrow 0
\end{equation}
inducing a $(k+1)$-step decreasing filtration on $\mathcal{H}_k$
\begin{equation}
    \mathcal{H}_k=\mathrm{Fil}^0\mathcal{H}_k\supseteq \mathrm{Fil}^1\mathcal{H}_k\supseteq\cdots\mathrm{Fil}^k\mathcal{H}_k\supseteq\mathrm{Fil}^{k+1}\mathcal{H}_k=0.
\end{equation}
The graded pieces of the filtration are isomorphic to
\begin{equation*}
    Gr_i\mathcal{H}_k=\mathrm{Fil}^{k-i}\mathcal{H}_k/\mathrm{Fil}^{k+1-i}\mathcal{H}_{k}\cong \omega^{k-2i}.
\end{equation*}
The connection obeys \emph{Griffith's transversality} 
\begin{equation*}
    \nabla \mathrm{Fil}^{k-i}\mathcal{H}_k\subseteq\mathrm{Fil}^{k-i-1}\mathcal{H}_k\otimes_{\mathcal{O}_X}\Omega^1_X(\log C)
\end{equation*}
and these maps induce a \emph{Kodaira-Spencer} type of isomorphism of $\mathcal{O}_X$-modules for $p>k+1$
\begin{equation}\label{Kodaira Spencer}
    \nabla:\mathrm{Gr}_i\mathcal{H}_k\cong \mathrm{Gr}_{i+1}\mathcal{H}_k\otimes_X\Omega^1_X(\log C).
\end{equation}
Following the usual notation, we write
\begin{align*}
    M_k(X)=\Gamma(X,\omega^{k}),\quad S_k(X)=\Gamma(X,\omega^{k-2}\otimes\Omega^1_X).
\end{align*}
We denote the complex arising from the connection by
\begin{equation*}
    \mathcal{H}_k^\bullet:=\left[ \mathcal{H}_k\rightarrow \mathcal{H}_k\otimes_{\mathcal{O}_X}\Omega^1_X\right]
\end{equation*}
and denote its hypercohomology groups by
\begin{equation*}
    H^i_{dR}(X,\mathcal{H}_k):=\mathbb{H}^i(X,\mathcal{H}_k^\bullet).
\end{equation*}
\begin{lemm}\label{lemma exact sequence theta quotient and relative de Rham}
    We have a short exact sequence
    \begin{equation}
    0\rightarrow \omega_{-k} \xrightarrow{\theta^{k+1}} \omega_{k+2}\rightarrow \mathcal{H}_k\otimes_X\Omega^1_X(\log C)/\nabla_k(\mathcal{H}_k)\rightarrow 0
\end{equation}
where the right map is compatible with inclusions and an \emph{Eichler-Shimura} short exact sequence of $\mathcal{O}_K$-modules
\begin{equation*}
    0\rightarrow M_{k+2}(X)\rightarrow H^{1}_{\mathrm{dR}}(X,\mathcal{H}_{k})\rightarrow S_{k+2}(X)^\vee\rightarrow0.
\end{equation*}
\end{lemm}
\begin{proof}
     Consider the filtration $(\mathrm{Fil}^p\mathcal{H}_k^\bullet)_p$ of the complex $\mathcal{H}_k^\bullet$ 
     \begin{equation*}\label{filtration relative de rham complex}
    \mathrm{Fil}^p\mathcal{H}_k^\bullet:=\left[\mathrm{Fil}^{p}\mathcal{H}_k\rightarrow \mathrm{Fil}^{p-1}\mathcal{H}_k\otimes_{\mathcal{O}_X}\Omega^1_X(\log C)\right].
\end{equation*}
and take the associated spectral sequence $E_0^{p,q}=\mathrm{Fil}^p\mathcal{H}_k^{p+q}/\mathrm{Fil}^{p+1}\mathcal{H}_k^{p+q}$. The only terms that survive in page 1 are
\begin{equation*}
    E_1^{0,0}\cong \mathcal{H}_k/\mathrm{Fil}^1\mathcal{H}_k,\quad E_1^{k+1,-k}\cong \mathrm{Fil}^{k}\mathcal{H}_k\otimes_X\Omega^1_X(\log C).
\end{equation*}
The spectral sequence degenerates at page $k+2$ converging to $E_{k+2}^{k+1,-k}=H^1(\mathcal{H}_k^\bullet)$. This gives rise to the short exact sequence
\begin{align*}
    0\rightarrow \mathcal{H}_k/\mathrm{Fil}^1\mathcal{H}_k \xrightarrow{\theta^{k+1}} \mathrm{Fil}^{k}\mathcal{H}_k\otimes_X\Omega^1_X(\log C)\rightarrow \mathcal{H}_k\otimes_X\Omega^1_X(\log C)/\nabla_k(\mathcal{H}_k)\rightarrow 0.
\end{align*}
The second exact sequence is obtained applying the spectral sequence to the hypercohomology of the graded pieces obtaining
\begin{align*}
    0&\rightarrow H_{\mathrm{dR}}^0(X,\mathcal{H}_k^\bullet)\rightarrow H^0(X,\omega^{-k})\rightarrow H^0(X,\omega^{k}\otimes_X\Omega^1_X(\log C))\rightarrow H_{\mathrm{dR}}^1(X,\mathcal{H}_k^\bullet)\rightarrow \\
    &\rightarrow H^1(X,\omega^{-k})\rightarrow H^1(X,\omega^{k}\otimes_X\Omega^1_X(\log C)))\rightarrow \cdots
\end{align*}
giving the desired exact sequence using ampleness of $\omega$ and Serre duality.
\end{proof}
Fix an effective divisor $S=\sum_{i=1}^{n}\alpha_i$ supported on the open curve $X\setminus C$ and defined over $\mathbb{Z}_p$,
\begin{equation*}
    \alpha_i:\mathrm{Spec}(\mathcal{O}_{K})\rightarrow X\setminus C
\end{equation*}
for $K$ a finite field extension over $\mathbb{Q}_p$ with distinct close points on the special fiber $\overline{X}$.
We denote by $\overline{S}$ their reduction to the special fiber $\overline{X}$ and $S_{\mathbb{Q}_p}$ their restriction to the generic fiber. We have that the curve $U=X_{\mathbb{Q}_p}\setminus S_{\mathbb{Q}_p}$ is affine.

As a direct consequence of the lemma, the de Rham cohomology of $\mathcal{H}_k$ can be computed by
\begin{equation*}
    H^1_{\mathrm{dR}}(U,\mathcal{H}_k)=M_{k+2}^{\mathrm{mero},S}/\theta^{k+1}(M_{-k}^{\mathrm{mero},S})
\end{equation*}
where we have used the notation $M_{k+2}^{\mathrm{mero},S}=\Gamma(U,\omega^{k+2})$ and should be thought of as \emph{meromorphic modular forms with poles along $S$}.

\section{Rigid cohomological description}

Consider $X^{\mathrm{rig}}$ to be the rigid analytification of $X$ and the reduction map
\begin{equation*}
    \mathrm{red}:X^{\mathrm{rig}}\rightarrow \overline{X}(\overline{\mathbb{F}}_p)
\end{equation*}
where we fixed $\overline{\mathbb{F}}_p$ algebraic closure of $\mathbb{F}_p$. We then consider the wide open obtained from $X^{\mathrm{rig}}$ removing closed balls in the residue disks of the points $S$. Since $X$ is smooth, $\mathrm{red}^{-1}(\overline\alpha)$ is conformal to an open ball $B$ for every $\overline{\alpha}\in \overline{S}$. Fixing this parametrization and fixing $0<\lambda<1$, define
\begin{equation*}
    V_{\lambda}=X^{\mathrm{rig}}\setminus\bigcup_{\alpha\in S}B(\alpha,\lambda)
\end{equation*}
with $B(\alpha,\lambda)$ closed ball inside $\mathrm{red}^{-1}(\overline{\alpha})$ centered in $\alpha$.
After an application of rigid GAGA to the coherent sheaves, the acyclicity of wide open neighborhoods implies that the hypercohomology of the complex $\mathcal{H}_k^{\bullet}$ can be computed simply by
    \begin{equation*}
        H^1_{\mathrm{dR}}(V_{\lambda},\mathcal{H}_k^{\mathrm{rig}})\cong
        \Gamma(V_\lambda, \mathcal{H}_k^{rig}\otimes\Omega^1_{X^{rig}}(\log C))/\nabla_k(\Gamma(V_\lambda, \mathcal{H}_k^{rig}))
    \end{equation*}

Borrowing once again the notation from the modular curve case, we will denote the sections of $\omega^{\mathrm{rig}}$ over $V_\lambda$ by
\begin{equation*}
    M_{k,\lambda}^{\dagger,S} = \Gamma(V_\lambda,(\omega^{\mathrm{rig}})^k)
\end{equation*}
and they should be thought as \emph{$\lambda$-overconvergent modular forms}. We will denote it simply as $M_{k,\lambda}^{\dagger}$ when the set of poles is clear from the context.

We can apply the acyclicity property of $V_\lambda$ to the short exact sequence of Lemma \ref{lemma exact sequence theta quotient and relative de Rham}.
    \begin{equation}\label{quotient of overconvergent modular forms as cohomology group}
    M_{k+2,\lambda}^\dagger/\theta^{k+1}(M_{-k,\lambda}^\dagger)\cong \Gamma(V_\lambda, \mathcal{H}_k^{rig}\otimes\Omega^1_{X^{rig}}(\log C))/\nabla_k(\Gamma(V_\lambda, \mathcal{H}_k^{rig}))\cong  H^1_{\mathrm{dR}}(V_{\lambda},\mathcal{H}_k^{\mathrm{rig}}).
    \end{equation}
    
Since every vector bundle with connection on a smooth formal model $X$ of $X^{rig}$ is overconvergent, the work of Baldassarri-Chiarellotto \cite{baldassarri-chiarellotto2001} establishes a comparison isomorphism
\begin{equation*}
    H^1_{\mathrm{dR}}((X_{\mathbb{Q}_p},S_{\mathbb{Q}_p}),\mathcal{H}_k)\cong H^1_{\mathrm{dR}}(V_{\lambda},\mathcal{H}_k^{\mathrm{rig}}).
\end{equation*}
The left-hand side in the previous isomorphism is the de Rham cohomology with logarithmic coefficients. In characteristic 0, by an analytic result of Deligne \cite{deligne1970equations} and later extended to an algebraic version in \cite{abbott‑kedlaya‑roe2007bounding} we can relate it to the open curve case
\begin{equation*}
    H^1_{\mathrm{dR}}(X_{\mathbb{Q}_p}\setminus S_{\mathbb{Q}_p},\mathcal{H}_k)\cong H^1_{\mathrm{dR}}((X_{\mathbb{Q}_p},S_{\mathbb{Q}_p}),\mathcal{H}_k)\cong H^1_{\mathrm{dR}}(V_{\lambda},\mathcal{H}_k^{\mathrm{rig}}).
\end{equation*}
    In particular, we have the following commutative diagram where all arrows are isomorphisms
    \begin{equation}\label{meromorphic to overconvergent isomorphism}
        \begin{tikzcd}
{M_{k+2}^{\mathrm{mero},S}/\theta^{k+1}(M_{-k,\lambda}^{\mathrm{mero},S})} \arrow[d] \arrow[rr] &  & {H^1_{\mathrm{dR}}(X_{\mathbb{Q}_p}\setminus S_{\mathbb{Q}_p},\mathcal{H}_k)} \arrow[d]   \\
{M_{k+2,\lambda}^\dagger/\theta^{k+1}(M_{-k,\lambda}^\dagger)} \arrow[rr]                       &  & {H^1_{\mathrm{dR}}(V_\lambda,\mathcal{H}_k^{\mathrm{rig}}).}
\end{tikzcd}
    \end{equation}
The natural inclusion $M_{k+2}^{\mathrm{mero},S}\hookrightarrow M_{k+2,\lambda}^{\dagger,S}$ allow us to represent the cohomology classes in the rigid setting by meromorphic modular poles along $S$.

In order to study and relate the Frobenius action on global sections and on the fibers we will use the structure of  \emph{overconvergent $\log$ $F$-isocrystal} of $(\mathcal{H}_k,\nabla)$. Explicit Frobenius lifts have been studied in the classical work of Katz \cite{katz1973padic} while the Shimura curve case has been studied by Kassaei in \cite{kassaei2004padic}.

We apply the Frobenius equivariant residue sequence following Coleman \cite{coleman1994padicShimura} and Coleman-Iovita \cite{coleman‑iovita2010hidden}.

\begin{equation}\label{Gysin sequence 1}
    0\rightarrow H^{1}_{\mathrm{dR}}(X^{rig},\mathcal{H}_k^{\mathrm{rig}})\rightarrow H^1_{\mathrm{dR}}(V_\lambda,\mathcal{H}_k^{\mathrm{rig}})\xrightarrow{\mathrm{Res}}\bigoplus_{\alpha\in S_K}H^0_{\mathrm{dR}}(A_{\alpha,\lambda},\mathcal{H}_{k})[1]\rightarrow0
\end{equation}
Where $A_{\alpha,\lambda}$ are rigid subspaces of $V_{\lambda}$ conformal to the open annuli $\mathrm{res}^{-1}(\overline{\alpha})\setminus B(\alpha,\lambda)$ and $[1]$ refers to the shift in the Frobenius action.
The previous exact sequence gives in fact a rigid analogue of the Gysin sequence. Since every residue disk admits a basis of horizontal section for an overconvergent isocrystal as proved by Katz in \cite{katz1971travaux}, we can rewrite the exact sequence as
\begin{equation}\label{Gysin sequence}
    0\rightarrow H^{1}_{\mathrm{dR}}(X^{rig},\mathcal{H}_k^{\mathrm{rig}})\rightarrow H^1_{\mathrm{dR}}(V_\lambda,\mathcal{H}_k^{\mathrm{rig}})\xrightarrow{\mathrm{Res}}\bigoplus_{\alpha\in S_{\mathbb{Q}_p}}\mathrm{Sym^k}(\alpha^*\mathcal{H})[1]\rightarrow 0.
\end{equation}

In particular we have that the cohomology groups have the following dimensions
\begin{itemize}
    \item $\dim_{\mathbb{Q}_p}H^{1}_{\mathrm{dR}}(X^{rig},\mathcal{H}_k^{\mathrm{rig}})=\dim_{\mathbb{Q}_p}M_{k+2}(X_{\mathbb{Q}_p})+\dim_{\mathbb{Q}_p}S_{k+2}(X_{\mathbb{Q}_p})$ by Lemma \ref{lemma exact sequence theta quotient and relative de Rham},
    \item $\dim_{\mathbb{Q}_p}\mathrm{Sym^k}(\alpha^*\mathcal{H})=k+1$ for every $\alpha\in S_{\mathbb{Q}_p}$,
    \item $\dim_{\mathbb{Q}_p}H^1_{\mathrm{dR}}(V_\lambda,\mathcal{H}_k^{\mathrm{rig}})=\dim_{\mathbb{Q}_p}H^{1}_{\mathrm{dR}}(X^{rig},\mathcal{H}_k^{\mathrm{rig}})+|S_{\mathbb{Q}_p}|(k+1)$.
\end{itemize}

In the elliptic setting we have 
\begin{equation*}
    \mathrm{Sym^k}(\alpha^*\mathcal{H})\cong \mathrm{Sym^k}(H^1_{\mathrm{dR}}(\mathcal{E}_\alpha))
\end{equation*} 
for $\mathcal{E}_\alpha$ elliptic curve corresponding to the point $\alpha$. In the quaternionic setting we have 
\begin{equation*}
    \mathrm{Sym^k}(\alpha^*\mathcal{H})\cong \mathrm{Sym^k}(e_1.H^1_{\mathrm{dR}}(\mathcal{A}_\alpha))
\end{equation*}
for $\mathcal{A}_\alpha$ abelian surface corresponding to the point $\alpha$ and $e_1$ a fixed choice of idempotent in $M_2(\mathbb{Z}_p)$.

\begin{rema}[Parallel transport of eigenbasis]
    The purpose of this remark is to explain how Frobenius eigenbases behave when the poles are moved inside the same residue disk. The rigid Gysin sequence \eqref{Gysin sequence 1} does not depend on the choice of poles as long as they are lifts of a set of point modulo $p$.
    Consider $S=\{\alpha_1,\dots,\alpha_n\}$ a set of smooth points on the modular curve $X$. Consider $S'=\{\alpha_1',\dots,\alpha_n'\}$ a second set for which $\alpha_i$ and $\alpha_i'$ lie in the same residue disk for $i=1,\dots,n$. 
    Recall that the 0-th hypercohomology group is given by
    \begin{equation*}
        H^0_{\mathrm{dR}}(A_{\alpha,\lambda},\mathcal{H}_{k})\cong \Gamma(A_{\alpha,\lambda},\mathcal{H}_k^{\nabla=0}).
    \end{equation*}
    Then we have the following commutative diagram

\begin{equation*}
    \begin{tikzcd}
0 \arrow[r] & {H_{\mathrm{dR}}^1(X_{\mathbb{Q}_p},\mathcal{H}_k)} \arrow[r]                                                   & {H_{\mathrm{dR}}^1(X_{\mathbb{Q}_p}\setminus S_{\mathbb{Q}_p},\mathcal{H}_k)} \arrow[r]                                 & \bigoplus_{\alpha\in S}\mathrm{Sym}^k(H^1_{\mathrm{dR}}(\mathcal{E}_{\alpha}))[1] \arrow[r]                                                   & 0 \\
0 \arrow[r] & {H^1_{\mathrm{dR}}(X^{\mathrm{rig}},\mathcal{H}_k^{\mathrm{rig}})} \arrow[r] \arrow[d] \arrow[u] & {H^1_{\mathrm{dR}}(V_\lambda,\mathcal{H}_k^{\mathrm{rig}})} \arrow[r] \arrow[d] \arrow[u] & {\bigoplus_{\alpha\in S}\Gamma(A_{\alpha,\lambda},\mathcal{H}_k^{\nabla=0})[1]} \arrow[r] \arrow[d] \arrow[u] & 0 \\
0 \arrow[r] & {H_{\mathrm{dR}}^1(X_{\mathbb{Q}_p},\mathcal{H}_k)} \arrow[r]                                                   & {H_{\mathrm{dR}}^1(X_{\mathbb{Q}_p}\setminus S'_{\mathbb{Q}_p},\mathcal{H}_k)} \arrow[r]                                & \bigoplus_{\alpha'\in S'}\mathrm{Sym}^k(H^1_{\mathrm{dR}}(\mathcal{E}_{\alpha'}))[1] \arrow[r]                                                & 0
\end{tikzcd}
\end{equation*}
where the vertical maps are isomorphisms. Fixing a basis for one set of poles $S$, we can move the basis to $S'$ using \emph{parallel transport} computing an horizontal basis
\begin{equation*}
    \chi(\alpha,\alpha'):\mathrm{Sym}^k(H^1_{\mathrm{dR}}(\mathcal{E}_{\alpha}))\rightarrow \mathrm{Sym}^k(H^1_{\mathrm{dR}}(\mathcal{E}_{\alpha'})).
\end{equation*}
For example, we could try to compute the eigenbasis on CM points and then move the basis to any other point in the same residue disk.
\end{rema}

\section{Integral Frobenius structure}
We now consider the $\log$-crystalline cohomology group $H^1_{\log-\mathrm{crys}}(((\overline{X},\overline{S})/\mathbb{Z}_p),\mathcal{H}_k)$. Since $\mathcal{H}_k$ has the structure of a $\log F$-crystal as studied by Ogus \cite{Ogus1994Fcrystals}, we obtain a $\mathbb{Z}_p$ module with a Frobenius action. By the classical work of Kato \cite{kato1989logarithmic} we have a comparison theorem with de Rham cohomology
\begin{equation*}
    H^1_{\log-\mathrm{crys}}(((\overline{X},\overline{S})/\mathbb{Z}_p),\mathcal{H}_k)\otimes \mathbb{Q}_p\cong H^1_{\mathrm{dR}}((X_{\mathbb{Q}_p},S_{\mathbb{Q}_p}),\mathcal{H}_k)
\end{equation*}
and a comparison with rigid cohomology
\begin{equation*}
    H^1_{\log-\mathrm{crys}}(((\overline{X},\overline{S})/\mathbb{Z}_p),\mathcal{H}_k)\otimes \mathbb{Q}_p\cong H^1_{\mathrm{dR}}(V_{\lambda},\mathcal{H}_k^{\mathrm{rig}})
\end{equation*}
compatible with the action of Frobenius. We then find a Frobenius stable $\mathbb{Z}_p$-lattice $\Lambda$ in the rigid cohomology
\begin{equation*}
    \Lambda:=\mathrm{Im}(H^1_{\log-\mathrm{crys}}((\overline{X}, \overline{S})/\mathbb{Z}_p,\mathcal{H}_k)\rightarrow H^1_{\log-\mathrm{crys}}((\overline{X}, \overline{S})/\mathbb{Z}_p,\mathcal{H}_k)\otimes_{\mathbb{Z}_p}\mathbb{Q}_p\cong H^1_{\mathrm{dR}}(V_{\lambda},\mathcal{H}_k^{\mathrm{rig}})).
\end{equation*}
By equation \eqref{quotient of overconvergent modular forms as cohomology group} we can see $\Lambda$ contained in
\begin{equation*}
    \Lambda\hookrightarrow M_{k+2,\lambda}^\dagger/\theta^{k+1}(M_{-k,\lambda}^\dagger).
\end{equation*}

\subsection{Explicit Frobenius lift}
Consider now the classical case of $S$ consisting of the set of supersingular points and $\lambda<1/p+1$. As explained by \cite{katz1973padic} and \cite{kassaei2004padic}, there exist a 
section of the projection map $X(p)\rightarrow X$ induced by the \emph{canonical subgroup}.
This morphism gives a lift of Frobenius $V$ on the modular curve $X$ given by
\begin{equation}\label{Frobenius lift not too supersingular locus}
    V:V_{\lambda}\rightarrow V_{\lambda^{p}}
\end{equation}
acting at the level of the moduli space sending the elliptic curve or abelian surface to the corresponding quotient by the canonical subgroup. In the case of the elliptic modular curve, the $V$ operator acts on $q$-expansion by
\begin{equation*}
    (Vf)(q^{1/N})=p^kf(q^{p/N}).
\end{equation*}

The morphism $V$ is flat and we can then define a $U_p$ operator as the trace of Frobenius acting on the de Rham cohomology $H^1_{dR}(V_{\lambda^{p}},\mathcal{H}_k)$. Coleman proved the following theorem in the case of the elliptic modular curve.
\begin{theo}[\cite{coleman1996classical} Theorem 5.4]\label{V and U_p inverse}
    The space $H^1_{\mathrm{dR}}(V_{\lambda},\mathcal{H}_k^{\mathrm{rig}})$ is finite dimensional and the operators $V$ and $U_p$ acting on the modular curve by correspondence induce endomorphisms $\mathrm{Frob}$ and $\mathrm{Ver}$ of $H^1_{\mathrm{dR}}(V_{\lambda},\mathcal{H}_k^{\mathrm{rig}})$ such that
    \begin{equation*}
        \mathrm{Frob}\circ\mathrm{Ver}= \mathrm{Ver}\circ\mathrm{Frob}=p^{k+1}.
    \end{equation*}
\end{theo}
As further explained, the action of Verschibung on the cohomology group $H^1_{\mathrm{dR}}(V_{\lambda^p},\mathcal{H}_k^{\mathrm{rig}})$ is given by
\begin{equation*}
    \mathrm{Ver}:[f]\rightarrow [U_pf]
\end{equation*}
for $f\in M_{k+2,\lambda}^\dagger$ under the isomorphism \eqref{quotient of overconvergent modular forms as cohomology group}
\begin{equation*}
    M_{k+2,\lambda}^\dagger/\theta^{k+1}(M_{-k,\lambda}^\dagger)\cong H^1_{\mathrm{dR}}(V_{\lambda},\mathcal{H}_k^{\mathrm{rig}}).
\end{equation*}
This relation extends formally from the results of Kassaei \cite{kassaei2004padic} to the Shimura curve over $\mathbb{Q}$. A more detailed proof of this statement will appear in PhD thesis of the author.

Enlarging the set $S$ and removing disks from the ordinary locus, we can restrict the action of Frobenius and $U_p$ to a smaller domain. The key easy lemma for Frobenius neighborhoods is the following
\begin{lemm}\label{lemma strict neighborhood}
    Let $B$ the residue disk of an ordinary point $\alpha$ defined over $\mathbb{F}_p$, let $t$ be a Serre-Tate parameter. Then $B\cong \mathrm{Spf}(\mathbb{Z}_{p}[[t]])_{\eta}$ and the Frobenius lift $V$ restricted to $B$ is a strict contraction.
\end{lemm}
Fix $\alpha$ to be an ordinary $\mathbb{Z}_p$ point on $X$ and $B\cong \mathrm{Spf}(\mathbb{Z}_{p}[[t]])_{\eta}$ its residue disk, where we have fixed the parametrization. Consider $V_{\lambda}$ the rigid variety obtained as before by removing small disks in the supersingular locus and $V$ the Frobenius lift obtained from the canonical subgroup. Then by the previous lemma, 
\begin{equation*}
V_{|B}:B\rightarrow B    
\end{equation*}
acts as a strict contraction. Removing a  small ball of radius $\lambda'$ containing $\alpha$ and the canonical lift of $\overline{\alpha}$, $V$ acts on the annulus contained in $B$ as
\begin{equation*}
    V:A_{\alpha,\lambda'}\rightarrow A_{\alpha,\lambda'^p}.
\end{equation*}
This means that we can remove from $V_{\lambda}$ a small ball in $B$ and several annuli around on the supersingular locus to obtain a wide open $V_{\lambda'}$ where the Frobenius lift $V$ acts as in \eqref{Frobenius lift not too supersingular locus}. 

Combining the results of this section we obtain that there exists $f_1,\dots,f_r\in M_{k+2,\lambda}^\dagger$ $\lambda$-overconvergent modular forms with $r=\dim_{\mathbb{Q}_p}H^1_{\mathrm{dR}}(V_{\lambda},\mathcal{H}_k^{\mathrm{rig}})$ such that 
\begin{equation*}
    \Lambda=\langle [f_i]:i=1,\dots,r\rangle_{\mathbb{Z}_p}
\end{equation*}
on which the $V$ operator acts as an endomorphism and for which we have the relation $V\circ U_p =U_p\circ V =p^{k+1}$. Under the chain of isomorphisms \eqref{meromorphic to overconvergent isomorphism}, we can pick $f_1,\dots,f_r$ to be meromorphic modular forms over $\mathbb{Q}_p$ with poles along $S$.

\begin{mainthm} \label{Theorem characteristic polynomial of Up action}
    Let $X$ be a modular curve (elliptic or quaternionic) of level $N$ defined over $\mathbb{Z}_p$, with $N>3$, $p>3$ and $(N,p)=1$. Fix $\alpha$ be a $\mathbb{Z}_p$ point on the modular curve. Let $k$ be a positive integer with $p>k+1$, let $P(X),Q(X)\in \mathbb{Q}_p[X]$ be respectively the characteristic polynomial for the Frobenius action on $\mathrm{Sym}^k(\alpha^*\mathcal{H})[1]$ and on $H^1_{\mathrm{dR}}(X^{\mathrm{rig}},\mathcal{H}_k^\mathrm{rig})$
    \begin{align*}
        P(X)=\det(\mathrm{Frob}-X&|\mathrm{Sym}^k(\alpha^*\mathcal{H})[1])=\sum_{i=0}^{k+1}c_iX^i,\\ Q(X)=\det(\mathrm{Frob}-X&|H^1_{\mathrm{dR}}(X^{\mathrm{rig}},\mathcal{H}_k))=\sum_{i=0}^{d}d_iX^i,\\
        P(X)Q(X)&=\sum_{i=0}^{d+k+1}e_iX^i
    \end{align*}
    with $d=\dim_{\mathbb{Q}_p}M_{k+2}+\dim_{\mathbb{Q}_p}S_{k+2}$. Let $M=d+k+1$, then for every $f\in M_{k+2}^{\mathrm{mero},\{\alpha\}}$ we have
    \begin{equation*}
        \sum_{i=0}^{M}e_{M-i}p^{M-i}U_p^i(f)\in \theta^{k+1}\left(M_{-k,\lambda}^{\dagger,\{\alpha\}}\right)
    \end{equation*}
    If we assume $P(X),Q(X)$ coprime, then there exist $f_1,\dots,f_{k+1}\in  M_{k+2}^{\mathrm{mero},\{\alpha\}}$ such that $[f_i]\in \Lambda$ for $i=1,\dots,k+1$ and 
    \begin{equation*}
        \sum_{i=0}^{k+1}c_ip^{k+1-i}U_p^i(f)\in \theta^{k+1}\left(M_{-k,\lambda}^{\dagger,\{\alpha\}}\right)
    \end{equation*}
    for every $f\in \langle f_1,\dots,f_{k+1}\rangle_{\mathbb{Z}_p}$.
\end{mainthm}
\begin{proof}
    Consider the Gysin exact sequence \eqref{Gysin sequence}
\begin{equation*}
    0\rightarrow H^{1}_{\mathrm{dR}}(X^{rig},\mathcal{H}_k^{\mathrm{rig}})\rightarrow H^1_{\mathrm{dR}}(V_\lambda,\mathcal{H}_k^{\mathrm{rig}})\xrightarrow{\mathrm{Res}}\mathrm{Sym^k}(\alpha^*\mathcal{H})[1]\rightarrow 0.
\end{equation*}    
    The assumption that the characteristic polynomials of Frobenius are coprime imply that the extension group of $\mathbb{Q}_p$-vector spaces with Frobenius action is zero and then the sequence split. In general, the product of the characteristic polynomials of Frobenius annihilates $H^1_{\mathrm{dR}}(V_\lambda,\mathcal{H}_k^{\mathrm{rig}})$. For every class $[f]\in H^1_{\mathrm{dR}}(V_\lambda,\mathcal{H}_k^{\mathrm{rig}})$ with $f\in M_{k+2}^{\mathrm{mero},\{\alpha\}}$ representative obtained under the chain of isomorphisms \eqref{meromorphic to overconvergent isomorphism} we then have
    \begin{align}\label{characteristic polynomial with Frobenius}
        \sum_{i=0}^{M}e_i\mathrm{Frob}^i([f])=0.
    \end{align}
    We can now use the trick of enlarging the set of removed points to allow an explicit lift of Frobenius. Consider $S=\{\alpha\}\cup SS_{p}$ where $SS_p$ is the set of supersingular points on $X$. Consider the following commutative diagram obtained by functoriality \eqref{Gysin sequence}
    \[
\adjustbox{scale=0.8,center}{%
\begin{tikzcd}
0 \arrow[r] & {H^1_{\mathrm{dR}}(X^{\mathrm{rig}},\mathcal{H}_k^{\mathrm{rig}})} \arrow[r] & {M_{k+2,\lambda}^{\dagger,S}/\theta^{k+1}\left(M_{-k,\lambda}^{\dagger,S}\right)} \arrow[r]              & {\bigoplus_{\alpha\in SS_{p,\mathbb{Q}_p}}\mathrm{Sym}^kH^1_{\mathrm{dR}}(\mathcal{E}_\alpha)\oplus \mathrm{Sym}^kH^1_{\mathrm{dR}}(\alpha^*\mathcal{E})[1]} \arrow[r] & 0\\
0 \arrow[r] & {H^1_{\mathrm{dR}}(X^{\mathrm{rig}},\mathcal{H}_k^{\mathrm{rig}})} \arrow[r] \arrow[u,"="] & {M_{k+2,\lambda}^{\dagger,\{\alpha\}}/\theta^{k+1}\left(M_{-k,\lambda}^{\dagger,\{\alpha\}}\right)} \arrow[r] \arrow[u]             & {\mathrm{Sym}^kH^1_{\mathrm{dR}}(\alpha^*\mathcal{E})[1]} \arrow[r] \arrow[u]& 0\\
\end{tikzcd}
}
\]
By Theorem \ref{V and U_p inverse} we have that Frobenius and Verschiebung act on $M_{k+2,\lambda}^{\dagger,S}/\theta^{k+1}\left(M_{-k,\lambda}^{\dagger,S}\right)$ by $V$ and $U_p$ respectively. Composing \eqref{characteristic polynomial with Frobenius} with $\mathrm{Ver}^{M}$ we obtain
\begin{align}
        \sum_{i=0}^{M}e_i(\mathrm{Ver}^{M}\circ\mathrm{Frob}^i)[f]=0,\\
        \sum_{j=0}^{M}e_{M-j}p^{M-j}U_p^{j}(f)\in \theta^{k+1}\left(M_{-k,\lambda}^{\dagger,\{\alpha\}}\right).
    \end{align}
In the case where the exact sequence splits, we consider a basis of elements coming from $\Lambda$ and we obtain the previous relation with the characteristic polynomial $P(X)=\det(\mathrm{Frob}-X|\mathrm{Sym}^k(\alpha^*\mathcal{H})[1])$.
\end{proof}

We conclude the section with two remarks. The first on a possible improvement using Hecke operators, the latter on known results about the slope of $p$-newforms.

\begin{rema}[Hecke operators]
    We could strengthen the previous theorem using the theory of Hecke operators. The geometric correspondences inducing Hecke operators act on the cohomology groups. By taking a linear combination $T=\sum\lambda_nT_n$, we can define an operator that annihilates the first classical term in the Gysin sequence \eqref{Gysin sequence}. In this way, the Frobenius submodule obtained from the cohomology group $H^1_{\mathrm{dR}}(V_\lambda,\mathcal{H}_k^{\mathrm{rig}})$ under the action of $T$ becomes isomorphic to the piece coming from the de Rham cohomology of the fibers.
\end{rema}

\begin{rema}[Slope of $p$-newforms]
    Consider once again the case of $S$ consisting of all the supersingular points of $X$. Then \cite{coleman1996classical} in the case of elliptic modular curves identified the parabolic part of $H^1_{\mathrm{dR}}(V_{p/p+1},\mathcal{H}_k)$ with classical modular forms of level $\Gamma_1(N)\cap\Gamma_0(p)$ and weight $k+2$. This has been later generalized in \cite{kassaei2009overconvergence} including the case of quaternionic modular forms. Theorem \ref{V and U_p inverse} allows us to relate the slopes of eigenvalues of the Frobenius with the slopes of eigenvalues of $U_p$ operator. In the elliptic case we obtain that the Frobenius module
    \begin{equation*}
        \bigoplus_{\alpha\in S_{\mathbb{Q}_p}}\mathrm{Sym^k}(\alpha^*\mathcal{H})\cong\bigoplus_{\alpha\in S_{\mathbb{Q}_p}}\mathrm{Sym}^kH^1_{\mathrm{dR}}(\mathcal{E}_{\alpha})
    \end{equation*}
    is isoclinic of slope $\frac{k}{2}$ where $\mathcal{E}_{\alpha}$ is the elliptic curve corresponding to the point $\alpha$. From the exact sequence \eqref{Gysin sequence} we deduce that the newspace of the classical modular forms of level $\Gamma_1(N)\cap\Gamma_0(p)$ is isoclinic of slope $\frac{k+2}{2}$ due to the $1$-shift and then slope $k/2$ due to the relation $V\circ U_p=p^{k+1}$.
\end{rema}

\section{Arithmetic applications}\label{arithmetic application section}
\subsection{Congruences on $q$-expansion}
In this section, we restrict to the case where $X=X_1(N)$ is an elliptic modular curve and we consider the expansion at the cusp $[\infty]$. We summarize here the classical geometric description of the $q$-expansion, referring to \cite{katz1973padic} for further details. Consider the infinitesimal neighborhood $\mathrm{Spf}(\mathbb{Z}_p[[q^{1/N}]])$ of the cusp $[\infty]$ of the formal model $\mathfrak{X}$. The fiber product with the universal elliptic curve gives rise to the Tate curve
\begin{equation*}
    \begin{tikzcd}
\mathrm{Tate}(q^{1/N}) \arrow[r] \arrow[d] & \mathcal{E} \arrow[d] \\
{\mathrm{Spf}(\mathbb{Z}_p[[q^{1/N}]])} \arrow[r,"\iota"]  & \mathfrak{X}.          
\end{tikzcd}
\end{equation*}
The global sections of the line bundle are given by
\begin{equation*}
    \Gamma(\mathrm{Spf}(\mathbb{Z}_p[[q^{1/N}]]),\iota^*\omega)=\mathbb{Z}_p[[q^{1/N}]]\omega_{\mathrm{can}}
\end{equation*}
where $\omega_{\mathrm{can}}$ is the canonical differential $dt/t$ on the Tate curve $\mathbb{G}_m/q^\mathbb{Z}$. 
The global sections of the sheaf of differential 1-forms on $X$ with logarithmic poles at the cusps are given by
\begin{equation*}
    \Gamma(\mathrm{Spf}(\mathbb{Z}_p[[q^{1/N}]]),\iota^*\Omega^1_X(\log C))=\mathbb{Z}_p[[q^{1/N}]]\frac{dq}{q}.
\end{equation*}
The Kodaira-Spencer isomorphism \eqref{Kodaira Spencer} gives an identification
\begin{equation*}
    KS(\omega_{\mathrm{can}}^2)=\frac{dq}{q}
\end{equation*}
Given a modular form $f\in \Gamma(U,\omega^k\otimes \Omega^1_{X}(\log C))$ with $U\subset X$ we can then define its $q$-expansion as
\begin{equation*}
    \iota^*f=f(q^{1/N})\omega_{\mathrm{can}}^k\otimes\frac{dq}{q},\quad f(q^{1/N})=\sum_{n=0}^\infty a_n(f)q^{\frac{n}{N}}.
\end{equation*}
The pullback of the relative de Rham cohomology at infinity is given by
\begin{equation*}
    \Gamma(\mathrm{Spf}(\mathbb{Z}_p[[q^{1/N}]]),\iota^*\mathcal{H})=\mathbb{Z}_p[[q^{1/N}]])\omega_{\mathrm{can}}\oplus\mathbb{Z}_p[[q^{1/N}]])\eta_{\mathrm{can}}
\end{equation*}
where 
\begin{align*}
\nabla(\omega_{\mathrm{can}})=:\eta_{\mathrm{can}}\otimes\frac{dq}{q}\in \iota^*(\mathcal{H}\otimes\Omega^{1}_X),\qquad
    \nabla(\eta_{\mathrm{can}})=0.
\end{align*}
In general, we will have
\begin{equation*}
    \Gamma(\mathrm{Spf}(\mathbb{Z}_p[[q^{1/N}]]),\iota^*\mathcal{H}_k)=\bigoplus_{a+b=k}\mathbb{Z}_p[[q^{1/N}]])\omega_{\mathrm{can}}^a\eta_{\mathrm{can}}^b
\end{equation*}
By the acyclicity of $\mathrm{Spf}(\mathbb{Z}_p[[q^{1/N}]])$, we obtain
\begin{equation*}
 H^1_{\mathrm{dR}}(\mathrm{Spf}(\mathbb{Z}_p[[q^{1/N}]]),\iota^*\mathcal{H}_k)=\Gamma(\mathrm{Spf}(\mathbb{Z}_p[[q^{1/N}]]),\iota^*(\mathcal{H}_k\otimes_{X}\Omega^1_{X}))/\nabla(\Gamma(\mathrm{Spf}(\mathbb{Z}_p[[q^{1/N}]]),\iota^*(\mathcal{H}_k))).
\end{equation*}
Using the exact sequence in Lemma \ref{lemma exact sequence theta quotient and relative de Rham} we then obtain
\begin{align*}
    H^1_{\mathrm{dR}}(\mathrm{Spf}(\mathbb{Z}_p[[q^{1/N}]]),\iota^*\mathcal{H}_k)&\cong \Gamma(\mathrm{Spf}(\mathbb{Z}_p[[q^{1/N}]]),\iota^*\omega^{k+2}/\theta^{k+1}(\Gamma(\mathrm{Spf}(\mathbb{Z}_p[[q^{1/N}]]),\iota^*\omega^{-k})))\\
    &\cong\mathbb{Z}_p[[q^{1/N}]]/\theta^{k+1}(\mathbb{Z}_p[[q^{1/N}]]).
\end{align*}
Combining the constructions, we can then form a \emph{$q$-expansion map} on the cohomology classes
\begin{lemm}
    The morphism $\iota:\mathrm{Spf}(\mathbb{Z}_p[[q^{1/N}]])\rightarrow \mathfrak{X}$ between formal schemes induces the following morphism
    \begin{align*}
        \Lambda&\rightarrow \mathbb{Z}_p[[q^{1/N}]]/\theta^{k+1}(\mathbb{Z}_p[[q^{1/N}]])\\
        [f]&\mapsto[f(q)].
    \end{align*}
\end{lemm}
\begin{proof}
    By an integral comparison with the de Rham cohomology \cite{kato1989logarithmic} (Theorem 6.4) we can compute the crystalline cohomology by
    \begin{equation*}
        H^1_{\log-\mathrm{crys}}((\overline{X}, \overline{S})/\mathbb{Z}_p,\mathcal{H}_k)\cong H^1_{\mathrm{dR}}((\mathfrak{X},S),\mathcal{H}_k).
    \end{equation*}
    The desired map is then induced by the functoriality along the morphism $\iota$
    \begin{equation*}
        H^1_{\mathrm{dR}}((\mathfrak{X},S),\mathcal{H}_k)\rightarrow H^1_{\mathrm{dR}}(\mathbb{Z}_p[[q^{1/N}]]),\iota^*\mathcal{H}_k).
    \end{equation*}
    By the previous remark, we have
    \begin{equation*}
        H^1_{\mathrm{dR}}(\mathbb{Z}_p[[q^{1/N}]]),\iota^*\mathcal{H}_k)\cong \mathbb{Z}_p[[q^{1/N}]]/\theta^{k+1}(\mathbb{Z}_p[[q^{1/N}]]).
    \end{equation*}
\end{proof}

\begin{coro}\label{congruence q-coefficients}
    Let $X$ be an elliptic modular curve of level $N$ defined over $\mathbb{Z}_p$, with $N>3$, $p>3$ and $(N,p)=1$. Fix a $\mathbb{Z}_p$-point $\alpha$ on the modular curve. Let $k$ be a positive integer with $p>k+1$, let $P(X),Q(X)\in \mathbb{Q}_p[X]$ be respectively the characteristic polynomials of the Frobenius action on $\mathrm{Sym}^kH^1_{\mathrm{dR}}(\alpha^*\mathcal{E})[1]$, with $\mathcal{E}\rightarrow X$ the universal elliptic curve, and on $H^1_{\mathrm{dR}}(X^{\mathrm{rig}},\mathcal{H}_k^\mathrm{rig})$
    \begin{align*}
        P(X)=\det(\mathrm{Frob}-X&|\mathrm{Sym}^k(\alpha^*\mathcal{H})[1])=\sum_{i=0}^{k+1}c_iX^i,\\ Q(X)=\det(\mathrm{Frob}-X&|H^1_{\mathrm{dR}}(X^{\mathrm{rig}},\mathcal{H}_k))=\sum_{i=0}^{d}d_iX^i,\\
        P(X)Q(X)&=\sum_{i=0}^{d+k+1}e_iX^i
    \end{align*}
    with $d=\dim_{\mathbb{Q}_p}H^1_{\mathrm{dR}}(X^{\mathrm{rig}},\mathcal{H}_k)=\dim_{\mathbb{Q}_p}M_{k+2}+\dim_{\mathbb{Q}_p}S_{k+2}$.
    Let $M=d+k+1$, then for every $f\in \Lambda$ we have
    \begin{equation*}
        \sum_{i=0}^{M}e_{M-i}p^{M-i}a_{np^{l+i}}(f)\equiv 0\mod p^{l(k+1)}.
    \end{equation*}
    If we assume that $P(X),Q(X)$ are coprime, then there exist $f_1,\dots,f_{k+1}\in  M_{k+2}^{\mathrm{mero},\{\alpha\}}$ such that $[f_i]\in \Lambda$ for $i=1,\dots,k+1$ and 
    \begin{equation*}
        \sum_{i=0}^{k+1}c_{k+1-i}p^{k+1-i}a_{np^{l+i}}(f)\equiv 0 \mod p^{l(k+1)}
    \end{equation*}
    for every $f\in \langle f_1,\dots,f_{k+1}\rangle_{\mathbb{Z}_p}$.
    
\end{coro}
\begin{proof}
    By Theorem \ref{Theorem characteristic polynomial of Up action} we have the following relation for every $f\in \Lambda$
    \begin{equation*}
         \sum_{i=0}^{M}e_{M-i}p^{M-i}U_p^i(f)\in \theta^{k+1}\left(M_{-k,\lambda}^{\dagger,\{\alpha\}}\right).
    \end{equation*}
    The action of $U_p$, which is the trace of the Frobenius lift $V$ on the $q$-expansion of weight $k+2$ modular forms, is classically given by
    \begin{equation*}
        f(q^{1/N})=\sum_{n\geq 0}a_n(f)q^{n/N}\mapsto (U_pf)(q^{1/N})=\sum_{n\geq 0}a_{pn}(f)q^{n/N}.
    \end{equation*}
In particular, using the previous Lemma, we obtain 
\begin{equation*}
    \sum_{i=0}^{M}e_{M-i}p^{M-i}U_p^i(f(q^{1/N}))\in \theta^{k+1}\left(\mathbb{Z}_p[[q^{1/N}]]\right)
\end{equation*}
from which we deduce for $l>1$
\begin{equation*}
   \sum_{i=0}^{M}e_{M-i}p^{M-i}a_{np^{l+i}}(f)\equiv 0 \mod p^{l(k+1)}.
\end{equation*}
\end{proof}

\subsection{An explicit example - Zhang's congruences}
In \cite{zhang2025ellipticcurvesfouriercoefficients}, Zhang provided numerous numerical observation about meromorphic modular forms and their relation with the Frobenius eigenvalues of the poles. In this short section we want to show how the previous constructions and results provide a framework to understand his congruences. 

We will work directly with the coarse moduli space $X(1)\cong \mathbb{P}^1_{j}$. The descent of the results from level $N$ to level $1$ can be carried over by classical techniques. We consider Example 4.6 of \emph{loc. cit.}. Fix $k=2$ and $p>k+1=3$, $C/\mathbb{Q}$ is the elliptic curve given by the equation
\begin{equation*}
    y^2+xy=x^3-x^2-2x-1.
\end{equation*}
We have $C$ is a CM elliptic curve with respect to the order $\mathbb{Z}[\frac{1+\sqrt{-7}}{2}]$ of discriminant -7 and $j(C)=-3375$. Consider $S=\{j(C)\}\cup SS_{p}$ where $SS_p$ consists in a set of lift of supersingular points on $X(1)$. The Gysin sequence \eqref{Gysin sequence} gives us the following commutative diagram

\adjustbox{scale=0.75,center}{%
\begin{tikzcd}
0 \arrow[r] & {H^1_{\mathrm{dR}}(X^{\mathrm{rig}},\mathcal{H}_2^{\mathrm{rig}})} \arrow[r] \arrow[d,"="]      & {M_{4,\lambda}^{\dagger,SS_p}/\theta^{3}(M_{-2,\lambda}^{\dagger,SS_p})} \arrow[r] \arrow[d] & {\bigoplus_{\alpha\in SS_{p,\mathbb{Q}_p}}\mathrm{Sym}^2H^1_{\mathrm{dR}}(\mathcal{E}_\alpha)[1]} \arrow[r] \arrow[d]      & 0 \\
0 \arrow[r] & {H^1_{\mathrm{dR}}(X^{\mathrm{rig}},\mathcal{H}_2^{\mathrm{rig}})} \arrow[r] & {M_{4,\lambda}^{\dagger,S}/\theta^{3}(M_{-k,\lambda}^{\dagger,S})} \arrow[r]              & {\bigoplus_{\alpha\in SS_{p,\mathbb{Q}_p}}\mathrm{Sym}^2H^1_{\mathrm{dR}}(\mathcal{E}_\alpha)\oplus \mathrm{Sym}^2H^1_{\mathrm{dR}}(C)[1]} \arrow[r] & 0\\
0 \arrow[r] & {H^1_{\mathrm{dR}}(X^{\mathrm{rig}},\mathcal{H}_2^{\mathrm{rig}})} \arrow[r] \arrow[u,"="] & {M_{4,\lambda}^{\dagger,\{j(C)\}}/\theta^{3}(M_{-2,\lambda}^{\dagger,\{j(C)\}})} \arrow[r] \arrow[u]             & {\mathrm{Sym}^2H^1_{\mathrm{dR}}(C)[1]} \arrow[r] \arrow[u]& 0\\
\end{tikzcd}
}
The radius $\lambda$ can be chosen accordingly to Lemma \ref{lemma strict neighborhood} to have an explicit lift of Frobenius corresponding with the quotient of the canonical subgroup.
The bottom row can be identified with the equivalent Gysin exact sequence in de Rham cohomology, in particular we have the isomorphism
\begin{equation*}
    M_{4}^{\mathrm{mero},\{j(C)\}}/\theta^{3}\left(M_{-2,\lambda}^{\mathrm{mero},\{j(C)\}}\right)\cong M_{4,\lambda}^{\dagger,\{j(C)\}}/\theta^{3}\left(M_{-2,\lambda}^{\dagger,\{j(C)\}}\right).
\end{equation*}
The three dimensional space given by $\mathrm{Sym}^2H^1_{\mathrm{dR}}(C)$ can be represented in the cohomology group by meromorphic modular forms with poles along $j(C)=-3375$. Zhang identifies the following basis
\begin{align*}
    f_1&=\frac{E_4}{j+3375},\\
    f_2&=19\cdot\frac{E_4}{j+3375}-91125\cdot\frac{E_4}{(j+3375)^2},\\
    f_3&=1399\cdot\frac{E_4}{j+3375}-19008675\cdot\frac{E_4}{(j+3375)^2}+54251268750\cdot\frac{E_4}{(j+3375)^3}.
\end{align*}
Zhang observed that for every ordinary prime $p>3$, for every $l>0$, we have
\begin{align*}
    a_{np^{l+1}}(f_1)&\equiv u_p(C)^2a_{np^{l}}(f_1)\mod p^{3l},\\
    a_{np^{l+1}}(f_2)&\equiv pa_{np^{l}}(f_2)\mod p^{3l},\\
    a_{np^{l+1}}(f_3)&\equiv p^2u_p(C)^{-2}a_{np^{l}}(f_3)\mod p^{3l}\\
\end{align*}
where $u_p(C)$ is the $p$-adic root of the polynomial $X^2-a_p(C)X+p$. According to Corollary \ref{congruence q-coefficients}, the meromorphic modular forms $f_1,f_2,f_3$ are the representatives of the classes in $H^1_{\mathrm{dR}}(V_{\lambda},\mathcal{H}_k^{\mathrm{rig}})$ corresponding to the Frobenius eigenspace isomorphic to $\mathrm{Sym}^2H^1_{\mathrm{dR}}(C)$.

\bibliographystyle{alpha}
\bibliography{bibliography}

\end{document}